\documentclass[pdflatex,sn-mathphys-num]{sn-jnl}

\usepackage{graphicx}%
\usepackage{multirow}%
\usepackage{amsmath,amssymb,amsfonts}%
\usepackage{amsthm}%
\usepackage{mathrsfs}%
\usepackage[title]{appendix}%
\usepackage{xcolor}%
\usepackage{textcomp}%
\usepackage{manyfoot}%
\usepackage{booktabs}%
\usepackage[shortlabels]{enumitem}
\usepackage{float}

\usepackage{indentfirst}

\newtheorem{thm}{Theorem}[section]
\newtheorem{prop}[thm]{Proposition}
\newtheorem{lemm}[thm]{Lemma}

\theoremstyle{definition}

\newtheorem{rem}[thm]{Remark}

\newtheorem{problem}{Problem}

\theoremstyle{remark}

\DeclareMathOperator{\closure}{cl}

\newcommand{\bR}{\mathbb{R}}

\newcommand{\bN}{\mathbb{N}}

\newcommand{\bL}{\mathbb{L}}

\newcommand*{\norm}[1]{\left\Vert #1\right\Vert} 
\newcommand*{\abs}[1]{\left\vert #1 \right\vert} 
\newcommand*{\set}[2]{\left\{#1\colon #2\right\}} 

\newcommand*{\mtr}[3][X]{\left\vert #2 - #3\right\vert_{#1}}

\newcommand*{\ball}[2][r]{B(#2, #1)}
\newcommand*{\cball}[2][r]{\overline{B}(#2, #1)}

\newcommand*{\seq}[3][1]{\left({#2}_{#3}\right)_{#3 = #1}^\infty}

\newcommand*{\ocinv}[2]{\left(#1\,;#2\right]}

\DeclareMathOperator{\Lip}{Lip}
\DeclareMathOperator{\lip}{lip}
\newcommand*{\LLip}{\mathop{\mathbb{L}\mathrm{ip}}}

\newcommand*{\lipnorm}[1]{\left\Vert #1\right\Vert_{\mathrm{lip}}}

\newcommand*{\rtext}[1]{\text{\color{red}}}

\begin{document}

\title[Centered Takagi--van der Waerden functions and their Lipschitz derivatives]
{Centered Takagi--van der Waerden functions and their Lipschitz derivatives}

\author*[1,2]{\fnm{Oleksandr V.} \sur{Maslyuchenko}}\email{ovmasl@gmail.com}
\equalcont{These authors contributed equally to this work.}

\author[1]{\fnm{Ziemowit M.} \sur{Wójcicki}}\email{ziemo1@onet.eu}
\equalcont{These authors contributed equally to this work.}

\affil*[1]{
    \orgdiv{Institute of Mathematics}, 
    \orgname{University of Silesia in Katowice}, 
    \orgaddress{
        \street{Bankowa 12}, 
        \city{Katowice}, 
        \postcode{40-007}, 
        \country{Poland}
        }
    }

\affil[2]{
    \orgdiv{Department of Mathematics and Informatics}, 
    \orgname{Yuriy Fedkovych Chernivtsi National University}, 
    \orgaddress{
        \street{Kotsiubynskoho 2}, 
        \city{Chernivtsi}, 
        \postcode{58012}, 
        \country{Ukraine}
    }
}

\abstract{
    Using a modification of a generalized Takagi-van der Waerden
    function on a metric space we prove that for any closed subset  
    of a metric space without isolated points there exists a continuous 
    function such that its  big and  local Lipschitz derivatives are 
    equal to infinity exactly on this set. Moreover, if given space is
    hermetic (for example, if it is normed) then the little  Lipschitz 
    derivative has the same property. 
    
}




\maketitle

\section{Introduction}
Famous Stepanoff's theorem \cite{St} asserts, that
any function $f\colon\bR^n\to\bR^m$ is differentiable
on almost every point of set $\set{x\in\bR^n}{\Lip f(x)<\infty}$,
where $\Lip f$ is so called big Lipschitz derivative of a function $f$.
Stepanoff's theorem is itself a generalization of Rademacher's theorem \cite{Radem, EvGa},
which states, that Lipschitz functions acting from an open subset of
$\bR^n$ to $\bR^m$, are almost everywhere differentiable.
Another two kinds of Lipschitz derivatives are known as the 
little Lipschitz derivative $\lip f$ and the local Lipschitz derivative $\LLip f$.
Strightforward connetction of Lipschitz derivatives to classical
notions of differentiability can be seen in the following fact:
if $f$ acts from an open subset $G$ of some normed space into an another normed space,
and there exists the Fr\'{e}chet derivative $f'$ of function $f$ then, indeed,
$\lip f(x)=\Lip f(x)=\norm{f'(x)}$ for all $x\in G$, see for example \cite{HerMa}.
It turns out that $\Lip f$ cannot be replaced by $\lip f$ in the Stepanoff's theorem,
as shown in \cite{BaloghCsornyei}. In general, working with $\lip f$
is more difficult, than with $\Lip f$.
Interest in the little Lipschitz derivative is of a different origin:
together with the big Lipschitz derivative they play a prominent
role in investigations of differentiable structures on measure
metric spaces, introduced by Cheeger in \cite{Cheg}.
They were particularly extensively used in follow-up
works by Keith in \cite{Keith1} and \cite{Keith2}.

    
Lately, there has been an interest in the following
problem: given a subset $A$ of a space $X$, is there
a continuous function $f$ defined on $X$, such that
Lipschitz derivatives of $f$ are infinite exactly on set $A$? 
For instance, in the paper \cite{BHRZ} of Buczolich, Hanson, Rmoutil, Z\"{u}rcher,
some partial answers to this problems were given for subsets of the real line.
In recent paper \cite{RmZu} Rmoutil and Z\"{u}rcher proved, that for any
$F_{\sigma\delta}$-set $A\subseteq\bR$ of Lebesgue measure zero, there exists
a continuous non-decreasing function $f\colon\bR\to\bR$ such that 
$\lip f(x)=\infty$ for $x\in A$ and $\lip f$ is finite otherwise.
On the other hand, in our previous paper \cite{MaWo}, we considered the case where $A=G$ 
is an open set in an arbitrary metric space and constructed a continuous function 
$f\colon X\to\bR$, such that $\Lip f=\LLip f=\infty$ exactly on $G$, by 
generalizing the classical nowhere-differentiable function of Takagi-van der Waerden.
Some additional assumptions on $X$ allowed us to
achieve also the equalities $\lip f=\Lip f=\LLip f=\infty$ strictly on $G$.

In this paper, we consider a closed subset $A$ of a metric space $X$
and we introduce the Takagi--van der Waerden function centered on $A$.
We prove, that if $A$ does not contain isolated points of $X$, then
$\Lip f(x)=\LLip f(x)=\infty$ if and only if $x\in A$, where $f$ is
the centered Takagi--van der Waerden function.
Moreover, if $X$ has some additional property (for example, if $X$ is a normed space),
then $\lip f=\Lip f=\LLip f=\infty$ exactly on $A$.

\section{Preliminary notions}
Throughout this paper $X$ will be a metric space. 
We denote by $\mtr\cdot\cdot$ the metric on $X$
and by $X^d$ the set of non-isolated points of
space $X$. By $\complement{A}$ for any set $A\subseteq X$
we always mean the complement of a set $A$ in a space $X$, i.e.,
$\complement{A}=X\setminus A$. 
The open and closed balls in $X$ with origin at $x\in X$ and radius $r>0$
will be denoted by $\ball[r]{x}$ and $\cball[r]{x}$ respectively.
The distance from a point $x$ to a set $A$ will be denoted with $d(x,A)$. 
The generalized ball is defined as $\ball[r]{A}=\big\{x\in X\colon d(x,A)<r\big\}$.
Through this paper, we will assume, that $\sup\varnothing = 0$.

For any function $f\colon X\to\bR$, we denote
\[
    \Lip^rf(x)=\sup_{u\in\ball[r]{x}}\frac{1}{r}\abs{f(x)-f(u)}, \; x\in X
\]
and
\[
    \lipnorm{f}=\sup_{\substack{u,v\in X \\ u\neq v}}\frac{\abs{f(u)-f(v)}}{\mtr{u}{v}}.
\]
Then, the \emph{big Lipschitz derivative} $\Lip f(x)$ of
a function $f$ at point $x\in X$ is defined by the formula
\begin{equation}\label{eq:bigLipDef}
    \Lip f(x)=\limsup_{r\to 0^{+}}\Lip^rf(x),
\end{equation}
and the \emph{little Lipschitz derivative} $\lip f(x)$ of a
function $f$ at point $x\in X$ is defined by the formula
\begin{equation}\label{eq:smallLipDef}
    \lip f(x)=\liminf_{r\to 0^{+}}\Lip^rf(x),
\end{equation}

The \emph{local Lipschitz derivative} $\LLip f(x)$ of a
function $f$ at point $x\in X$ is defined by the formula
\begin{equation}\label{eq:localLipDef}
    \LLip f(x)=\limsup_{(u,v)\to(x,x)}\frac{\abs{f(u)-f(v)}}{\mtr{u}{v}} = \inf_{r>0}\lipnorm{f\big\vert_{\ball[r]{x}}}.
\end{equation}
Since we have assumed that $\sup\varnothing=0$, 
the equality $\lip f(x)=\Lip f(x)=\LLip f(x)=0$ holds for any
isolated point $x$. 

Let us note the important identity which will be often useful:
\begin{equation}\label{eq:bigLipEasierDef}
    \Lip f(x)=\limsup_{u\to x}\frac{\abs{f(u)-f(x)}}{\mtr[X]{u}{x}}.
\end{equation}
Indeed, the equation \eqref{eq:bigLipEasierDef} is 
often used as definition of the big Lipschitz derivative.
For proof of equivalence of \eqref{eq:bigLipEasierDef} 
and \eqref{eq:bigLipDef} see \cite[Proposition 2.1]{MaWo}.
Analogical equality for $\lip f$ does not hold.
To be more precise,
\[
\liminf\limits_{u\to x}\frac{\abs{f(u)-f(x)}}{\mtr[X]{u}{x}}\leq \lip f(x)
\]
and the inequality can be strict.

Now we define certain special sets associated with a function $f$ as follows
\begin{itemize}
    \item $L^{\infty}(f)=\set{x\in X}{\Lip f(x)=\infty}$,
    \item $\ell^{\infty}(f)=\set{x\in X}{\lip f(x)=\infty}$,
    \item $\bL^{\infty}(f)=\set{x\in X}{\LLip f(x)=\infty}$.
\end{itemize}
It is important to note that passing to an equivalent metric
can change the sets $L^{\infty}(f)$, $\ell^{\infty}(f)$ and $\bL^{\infty}(f)$.
Take $X=\bR$, $f(x)=x$, $x\in X$ and consider metric given as 
$\mtr[1]{x}{y}=\abs{x^3-y^3}$, $x,y\in\bR$.
Obviously, metric $\mtr[1]{\cdot}{\cdot}$ is equivalent to standard Euclidean metric on $\bR$.
In Euclidean metric, we have $\lip f(0)=1$.
On the other hand, $\lip f(0)=+\infty$, when metric $\mtr[1]{\cdot}{\cdot}$ is considered.
In the first case, $0\notin\ell^{\infty}(f)$. In the second case, $0\in\ell^{\infty}(f)$.
\begin{prop}[{\cite[Proposition 3.4.]{MaWo}}]\label{prop:Linfty_sets_closed}
    For any continuous function ${f\colon X\to\bR}$ the set $\bL^{\infty}(f)$
    is closed,  $L^{\infty}(f)$ is a $G_\delta$-set in $X$ and 
    $\ell^\infty(f)$ is an $F_{\sigma\delta}$-set in $X$.
\end{prop}

Hereby, we consider the following 
\begin{problem}[$\ell$-$L$-$\bL$-problem]
    Is a subset $A$ of metric space closed 
    if and only if there exists a continuous function $f\colon X\to\bR$,
    such that $\ell^{\infty}(f)=L^{\infty}(f)=\bL^{\infty}(f)$? 
\end{problem}

We will at first tackle the following partial case:
\begin{problem}[$L$-$\bL$-problem]
    Let $X$ be a metric space and $A\subseteq X^d$. 
    Is $A$ closed if and only if there exists a 
    continuous function $f\colon X\to\bR$,
    such that $L^{\infty}(f)=\bL^{\infty}(f)$? 
\end{problem}

\section{Centered Takagi-van der Waerden functions}

\paragraph{Maximal separated sets.}
A subset $S$ of a metric space $X$ is called
\begin{itemize}
    \item \emph{$\varepsilon$-dense} in $X$ if for any $x\in X$
    there exists $s\in S$ such that $\mtr{x}{s}<\varepsilon$.
    \item \emph{$\varepsilon$-separated} if $\mtr{s}{t}\geq\varepsilon$
for any $s,t\in S$.
    \item \emph{maximal $\varepsilon$-separated} if it is $\varepsilon$-separated 
    in $X$ and for any $\varepsilon$-separated set $T\supseteq S$ we have that $T=S$.
\end{itemize}

\begin{rem}\label{rem:epsilon_dense_lip}
    For any subset $S$ of a metric space $X$
    the function ${x\mapsto d(x,S)}$ is $1$-Lipschitz.
    Moreover, $S$ is $\varepsilon$-dense if and only if $d(x,S)<\varepsilon$ for any $x\in X$.
\end{rem}

Using Teichm\"{u}ller-Tukey Lemma or equivalent tools
(like Kuratowski--Zorn Lemma) we can easily prove the following
\begin{prop}
    For any $\varepsilon$-separated set $S_0$ in $X$ there
    exists a maximal $\varepsilon$-separated set $S$ in $X$
    containing $S_0$.
\end{prop}

Another easy but important property 
of maximal $\varepsilon$-separated sets
is given by the following
\begin{prop}
    Any $\varepsilon$-separated subset $S$ of a space $X$ is
    a maximal $\varepsilon$-separated set in $X$ if
    and only if $S$ is $\varepsilon$-dense in $X$.
\end{prop}


\paragraph{Definition of centered Takagi--van der Waerden function.}

Let $A\subseteq X$ be a closed set, $a>b\geq 1$ and $c>0$ be some real numbers.
We define a \emph{Takagi--van der Waerden function of a type $(a,b,c)$
centered at a set $A$} as a function $f\colon X\to\bR$ such that
\[
    f(x)=\sum_{n=0}^{\infty}b^nd(x,T_n), \; x\in X,
\]
where 
$S_n$ is a maximal $\frac{1}{a^n}$-separated set in $G_n=\ball[\frac{c}{a^n}]{A}$ and $T_n=\complement{G_n}\cup S_n$ for any  $n\in\bN$.
We will also  shortly call $f$ a $TW$-function of a given type.
Note, that the sets $T_n$ defined above are $\frac{1}{a^n}$-dense in $X$.
Furthermore, if the sequence of sets $\seq[0]{T}{n}$ is ascending,
we will say that $f$ is \emph{Takagi-van der Waerden function 
of a monotone type $(a,b,c)$ centered at $A$}.

In fact, we can always construct a sequence 
$\seq[0]{T}{n}$ recursively to be ascending.
Denote $F_n=\complement{G_n}$.
Obviously,  $F_n\subseteq F_{n+1}$.
If $S_n$ is $\frac{1}{a^n}$-separated in $G_n$ then $S_n'=S_n\cap G_{n+1}$ is $\frac{1}{a^{n+1}}$-separated in $G_{n+1}$. 
Therefore, there exists a maximal  $\frac{1}{a^{n+1}}$-separated set $S_{n+1}\supseteq S_n'$ in $G_{n+1}$. 
Thus, 
\begin{align*}
    T_n=F_n\cup S_n&\subseteq F_{n+1}\cup S_{n}\\
    &= F_{n+1}\cup (S_{n}\setminus F_{n+1})=F_{n+1}\cup(S_n\cap G_{n+1})\\
    &=F_{n+1}\cup S_n'\subseteq F_{n+1}\cup S_{n+1}=T_{n+1}.
\end{align*}

We will denote by $s_n$ and $r_n$ be the $n$-th 
partial sum and the $n$-th reminder of $f$ respectively:
\[
    s_n(x)=\sum_{k=0}^{n-1}b^kd(x,T_k) \text{ and } r_n(x)=\sum_{k=n}^{\infty}b^kd(x,T_k), \text{ for } x\in X.
\]

\begin{prop}\label{prop:TWfunction_is_Lipschitz}
	Let $X$ be a metric space and let $a>b\geq 1$, $c>0$. 
    Assume, that $f$ is TW-function of type $(a,b,c)$
    centered at a closed set $A\subseteq X$.
    Then for any $n\in\bN$ the following 
    conditions hold:
    \begin{enumerate}[label=\normalfont{(\roman*)}]
		\item $f\colon X\to\mathbb{R}$ is a continuous function such that $0\le f(x)\le\frac a{a-b}$, $x\in X$;
		\item $r_n\colon X\to\mathbb{R}$ is a continuous function such that $0\le r_n(x)\le\frac {b^{n}}{(a-b)a^{n-1}}$, $x\in X$.
            Moreover, we have $r_n(x)=0$ for all $x\in F_n=\complement{\ball[c/a^n]{A}}$;
		\item if $b>1$ then $s_n\colon X\to\mathbb{R}$ is a Lipschitz function with the constant $\frac{b^n}{b-1}$.
	\end{enumerate} 
\end{prop}
\begin{proof}
    All conditions (i)-(iii) are trivially true for an isolated point $x$.
    Thus, in the following we consider $x\in X^d$.
    
	(ii). 
    Fix $n\in\bN_0$. If $x\in F_n$, then $d(x,T_k)=0$ since $F_n\subseteq F_k\subseteq T_k$
    for any $k\geq n$. Thus, $r_n(x)=0$.

    Consider $k\in\bN_0$ and let $x\in G_k$. 
    Therefore, $d(x,S_k)<\frac{1}{a^k}$, by Remark~\ref{rem:epsilon_dense_lip}.
    Since $S_k\subseteq T_k$, we have
    \[
        b^kd(x,T_k)=b^k\inf_{t\in T_k}\mtr{x}{t}\leq b^k\inf_{s\in S_k}\mtr{x}{s}=b^kd(x,S_k)<\left(\frac{b}{a}\right)^k.
    \]
    In the case when $x\notin G_k$ we have $b_kd(x,T_k)=0<\left(\frac{b}{a}\right)^k$. 
    Then, we conclude that
    \[
        0\leq\sum_{k=n}^\infty b^kd(x,T_k)\leq\sum_{k=n}^\infty\left(\frac{b}{a}\right)^k<\infty,
    \]
    so the series $r_n(x)=\sum_{k=n}^\infty b^kd(x,T_k)$ is uniformly convergent for any $x$
    and as a result, function $r_n$ is continuous.
    Moreover, we have the following estimation
    \[
        r_n(x)\leq\sum_{k=n}^\infty\left(\frac{b}{a}\right)^k=\frac{\left(\frac{b}{a}\right)^n}{1-\frac{b}{a}}=\frac{b^n}{(a-b)b^{n-1}}.
    \]

	(i). From $f=r_0$ we conclude that $(ii)\Rightarrow(i)$.
	
    (iii). By Remark \ref{rem:epsilon_dense_lip} we have
    \[
        \vert s_n(x)-s_n(y)\vert\leq \sum\limits_{k=0}^{n-1}b^k\big\vert d(x,S_k)-d(y,S_k)\big\vert
	\leq
    \]
    \[
        \leq\sum\limits_{k=0}^{n-1}b^k\mtr{x}{y}=\frac{b^n-1}{b-1}\mtr{x}{y}\leq \frac{b^n}{b-1}\mtr{x}{y}
    \]
	for any $x,y\in X$.
\end{proof}

\section{Solution of $L$-$\bL$-problem}

\begin{thm}\label{thm:first_problem}
    Let $X$ be a metric space, $A\subseteq X^d$ be a closed set,
    $a>b>2$, $c\geq 1$ be real numbers, and $f\colon X\to\bR$ be a TW-function 
    of monotone type $(a,b,c)$ centered at $A$. Then 
    $L^{\infty}(f)=\bL^{\infty}(f)=A$.
\end{thm}
\begin{proof}
    We will denote $\varphi_n(x)=d(x,T_n)$.
    Fix $x\in A$. We consider two cases.

    \textit{The first case:} $\displaystyle x\in\bigcup_{n=0}^\infty S_n$.
    Let $n_0$ be such that $x\in S_{n_0}$.
    Since $x\in A\subseteq X^d$, for any $n\in\bN$ there exists
    $u_n\in\ball[\frac{1}{2a^n}]{x}\setminus\{x\}$. Denote $\rho_n=\mtr{u_n}{x}$, $n\in\bN$.
    We have $0<\rho_n<\frac{1}{2a^n}$ and $a>1$, so $\displaystyle\lim_{n\to\infty}\rho_n=0$.
    Moreover, for any $n\geq n_0$ we have $x\in S_{n_0}\subseteq T_{n_0}\subseteq T_n$,
    which implies that $\varphi_n(x)=0$. 
    Hence, we obtain the equality
    \begin{equation}\label{eq:qwerty}
        r_n(x)=\sum_{k=n}^{\infty}b^k\varphi_k(x)=0, \; n\geq n_0.
    \end{equation}
    Fix $n\geq n_0$.
    We will show, that $\varphi_n(u_n)=\rho_n$. 
    Let $t\in T_n$. Then, either $t\in\complement{G_n}$ or $t\in S_n$.
    Clearly, in the first case, we have $\mtr{t}{x}\geq\frac{c}{a^n}\geq\frac{1}{a^n}$.
    In the second case we also have $\mtr{t}{x}\geq\frac{1}{a^n}$, 
    since $S_n$ is $\frac{1}{a^n}$-separated and $t\neq x\in T_n\cap G_n=S_n$.
    For any $t\in T_n\setminus\{x\}$ we have the following estimations
    \begin{align*}
        \mtr{u_n}{t} &\geq \mtr{t}{x}-\mtr{u_n}{x} \\
        &\geq \frac{1}{a^n}-\frac{1}{2a^n}=\frac{1}{2a^n}>\rho_n=\mtr{u_n}{x}.
    \end{align*}
    This shows, that 
    $\varphi_n(u_n)=\min\limits_{t\in T_n}\mtr{u_n}{t}=\mtr{u_n}{x}=\rho_n$.
    Note, that
    \[
        r_n(u_n)=\sum_{k=n}^\infty b^k\varphi_k(u_n)\geq b^n\varphi_n(u_n)=b^n\rho_n,
    \]
    which yields the following inequality
    \begin{equation}\label{eq:asdf}
        \abs{r_n(u_n)-r_n(x)}=r_n(u_n)\geq b^n\rho_n.
    \end{equation}
    Next, since 
    $s_n$ is a $\frac{b^n}{b-1}$-Lipschitz function 
    by Proposition \ref{prop:TWfunction_is_Lipschitz}$(iii)$,
    we conclude  
    \begin{equation}\label{eq:zxcv}
        \abs{s_n(u_n)-s_n(x)}\leq\frac{b^n}{b-1}\mtr{u_n}{x}=\frac{b^n\rho_n}{b-1}.
    \end{equation}
    Now, we obtain
    \begin{align*}
        \abs{f(u_n)-f(x)} &= \abs{\big(r_n(u_n)+s_n(u_n)\big)-\big(r_n(x)+s_n(x)\big)} \\
        &= \abs{\big(r_n(u_n)-r_n(x)\big)+\big(s_n(u_n)-s_n(x)\big)} \\
        &\geq \abs{r_n(u_n)-r_n(x)} - \abs{s_n(u_n)-s_n(x)} \\
        &\geq b^n\rho_n-\frac{b^n\rho_n}{b-1}=\frac{b-2}{b-1}b^n\rho_n.
    \end{align*}
    Finally, we have 
    $$\displaystyle\frac{1}{\rho_n}\abs{f(u_n)-f(x)}\geq\frac{b-2}{b-1}b^n\xrightarrow{n\to\infty}\infty,$$
    because $b>2$.
    We have found the sequence $\seq[0]{u}{n}$, such that $u_n\to x$ and
    \[
        \lim_{n\to\infty}\frac{\abs{f(u_n)-f(x)}}{\mtr{u_n}{x}}=\infty,
    \]
    which means that 
    \[
        \Lip f(x)=\limsup_{u\to x}\frac{\abs{f(u)-f(x)}}{\mtr{u}{x}}=\infty.
    \]

    \textit{The second case:} $\displaystyle x\notin\bigcup_{n=0}^\infty S_n$.  Fix $n\in\bN$.
    So, $x\in G_n=\ball[c/a^n]{A}$ and $x\notin S_n$.
    Hence, $x\notin\complement G_n\cup S_n= T_n$.
    Then $\varphi_n(x)>0$. Since $\frac{b}{2}>1$,
    we have $\inf\limits_{t\in T_n}\mtr{x}{t}=\varphi_n(x)<\frac{b}{2}\varphi_n(x)$.
    So, there exists $t_n\in T_n$ such that $\rho_n=\mtr{t_n}{x}<\frac{b}{2}\varphi_n(x)$.
    Since $T_n$ is an $\frac{1}{a^n}$-dense set,
    Remark \ref{rem:epsilon_dense_lip} yields the inequality
    $\varphi_n(x)<\frac{1}{a_n}$.
    Then $\rho_n<\frac{b}{2a^n}$.
    Next, observe that $\varphi_k(t_n)=0$ for $k\geq n$, since
    $t_n\in T_n\subseteq T_k$. Therefore,
    $r_n(t_n)=0$.
    Similarly to the previous case, we obtain
    \begin{align*}
        \abs{f(x)-f(t_n)} &= \abs{r_n(x)+s_n(x)-r_n(t_n)-s_n(x)} \\
        &\geq r_n(x)-\abs{s_n(x)-s_n(t_n)} \\
        &\geq b^n\varphi_n(x)-\frac{b^n}{b-1}\mtr{x}{t_n} \\
        &\geq b^n\rho_n\frac{2}{b}-\frac{b^n\rho_n}{b-1}=\frac{b-2}{b(b-1)}b^n\rho_n.
    \end{align*}
    Thus, $\mtr{x}{t_n}=\rho_n\to 0$ and $\frac{1}{\rho_n}\abs{f(x)-f(t_n)}\geq\frac{b-2}{b(b-1)}b^n\to\infty$.
    Again, we have $\Lip f(x)=\infty$.

    Thus far, combining both cases,
    we have shown that $A\subseteq L^{\infty}(f)$. 
    Then $A\subseteq L^{\infty}(f)\subseteq\bL^{\infty}(f)$.
    We will prove, that $\bL^{\infty}(f)\subseteq A$.
    Consider $x\in X\setminus A$.
    Then, there exists $n_0\in\bN_0$ such that
    $x\notin\closure{G_{n_0}}$. The set $U=\complement\closure{G_{n_0}}$ is
    a neighborhood of the point $x$ and for any $u\in U$ we have
    $u\in\complement{\closure{G_{n_0}}}\subseteq T_{n_0}\subseteq T_k$,
    so $\varphi_k(u)=0$ for any $k\geq n_0$.
    This means that $f(u)=s_{n_0}(u)$ for any $u\in U$. 
    Therefore, by Proposition~\ref{prop:TWfunction_is_Lipschitz}$(iii)$ we conclude that
    \[\LLip f(x)=\LLip s_{n_0}(x)\leq\dfrac{b^{n_0}}{b-1}<\infty.\] 
    This means, that $x\in X\setminus\bL^{\infty}(f)$.
    We have shown that $X\setminus A\subseteq X\setminus\bL^{\infty}(f)$,
    therefore $\bL^{\infty}(f)\subseteq A$. 
    Thus, $A=L^{\infty}(f)=\bL^{\infty}(f)$.
\end{proof}


The previous theorem allow us to give a solution
of the $L$-$\bL$-problem.
\begin{thm}
    Let $A$ be a subset of a metric space $X$. 
    Then, the following
    conditions are equivalent:
    \begin{enumerate}[label=\normalfont{(\roman*)}]
        \item $A$ is a closed set which does not contain isolated points 
             of $X$, \label{it:A_is_perfect}
        \item there exists a continuous function $f\colon X\to\bR$
        such that \[\bL^{\infty}(f)=L^{\infty}(f)=A.\] \label{it:A_is_2Lf_set}
    \end{enumerate}
\end{thm}
\begin{proof}
    Follows from Proposition~\ref{prop:Linfty_sets_closed}
    and Theorem~\ref{thm:first_problem}.
\end{proof}


\section{Solution of $\ell$-$L$-$\bL$-problem in a hermetic space}

The \emph{hermeticity of a metric space $X$ at a point $x\in X^d$} 
was defined in \cite{MaWo} by the formula
\[
    H(X,x)=\liminf_{r\to 0^{+}}\frac{1}{r}\sup_{u\in\ball[r]{x}}\mtr{u}{x},
\]
and the \emph{hermeticity of $X$} as the number $H(X)=\inf\limits_{x\in X^d}H(X,x)$.
We say, that $X$ is hermetic if $H(X)>0$.

If we denote by $d_x(u):=\mtr{x}{u}$, then $H(X,x)=\lip d_x(x)$, $x\in X$.
As a result, we get
\begin{rem}\label{rem:hermetic_as_well}
    Let $Y$ be an open subspace of a metric space $X$
    and $x\in Y$. Then $H(x,Y)=H(x,X)$. In particular,
    an open subset of a hermetic space is hermetic as
    well.
\end{rem}

The next proposition gives a geometric intuition behind hermeticity.

\begin{prop}[{\cite[Proposition 6.2]{MaWo}}]\label{prop:hermeticity_equivalent}
    Let $X$ be a hermetic metric space, 
    $x\in X$ and $0<\lambda<H(X,x)$. 
    Then there exists a number $R=R(X,x,\lambda)>0$ such that
    for any $r\in\ocinv{0}{R}$ there exists $u\in X$ 
    with
    \begin{equation}\label{eq:hermeticity_equivalent}
        \lambda r\leq\mtr{x}{u}\leq r.
    \end{equation}
\end{prop}

The following lemma is important for the proof of our main theorem.
\begin{lemm}[{\cite[Lemma 6.4]{MaWo}}]\label{lemm:v}
    Suppose that $Y$ is a metric space, $x\in Y^d$ and $H(Y,x)>0$.
    Fix real numbers $0<\lambda<H(Y,x)$ and $0<\varepsilon\leq R$, where $R=R(Y,x,\lambda)$
    is the number as in Proposition \ref{prop:hermeticity_equivalent}.
    Moreover, let $E$ be a maximal $\varepsilon$-separated set in $Y$
    and denote $\psi(y)=d(y,E)$ for any $y\in Y$.
    Then, there exists $u\in\cball[\varepsilon]{x}$ such that
    \[
        \abs{\psi(u)-\psi(x)}\geq\frac{\lambda\varepsilon}{8}.
    \]
\end{lemm}
Now we can solve $\ell$-$L$-$\bL$ problem for a hermetic space.
\begin{thm}
    Let $X$ be a hermetic metric space
    and let $A\subseteq X^d$ be a closed subset of $X$.
    Let $a>b>1+\frac{8}{H(X)}$ be real numbers. 
    Then, there exists $c>0$ such that for
    $TW$-function $f\colon X\to\bR$ of type $(a,b,c)$
    centered at set $A$, equalities
    $\ell^{\infty}(f)=L^{\infty}(f)=\bL^{\infty}(f)=A$
    hold.
\end{thm}
\begin{proof} Observe that $\frac{1}{b-1}<\frac{H(X)}{8}$.
Let $\alpha>0$ be such that
\[
    \frac{1}{b-1}+\frac{2b}{\alpha(a-b)}<\frac{H(X)}{8}.
\]
Then, there exists $\lambda<H(X)$ such that
\[
    \frac{1}{b-1}+\frac{2b}{\alpha(a-b)}<\frac{\lambda}{8}.
\]

Choose $c\geq \alpha+1$ and let $f$ be TW-function of $(a,b,c)$-type, centered at set $A$.
Fix $x\in A$. We shall prove, that $\lip f(x)=\infty$.

Note that $\lambda<H(X)\le H(X,x)$. 
Let $R=R(X,x,\lambda)>0$ be chosen as is  Proposition~\ref{prop:hermeticity_equivalent}.
Put $\varepsilon_n=\frac{\alpha}{a^n}$ for $n\in\bN$ and 
$\delta=\min\left\{\alpha,R\right\}$.
For any $0<r<\delta$  there exists
$n(r)\in\bN$ such that $\varepsilon_{n(r)}\leq r<\varepsilon_{n(r)-1}$.
Therefore, $\lim\limits_{r\to 0+}\varepsilon_{n(r)}=0$ and so, $\lim\limits_{r\to 0+}n(r)=\infty$. 

Fix $0<r<\delta$. To simplify, we will denote $n=n(r)$.
Let $Y=G_n=\ball[\varepsilon_n]{A}$, 
$E=S_n$, $\psi=d(\,\cdot\, ,S_n)$ and $\varepsilon=\varepsilon_n$.
Since $Y$ is an open subspace of $X$, by Remark~\ref{rem:hermetic_as_well}
we have $\lambda\leq H(X,x)=H(Y,x)$. 
Set $R_Y=\varepsilon$. Then $\ball[R_Y]{x}\subseteq Y$. Since $R_Y\in(0;R]$, the number $R_Y=R(Y,\lambda,x)$ satisfies the thesis of Proposition~\ref{prop:hermeticity_equivalent} in the space $Y$.
Now, we apply Lemma \ref{lemm:v} for $Y$, $E$, $\varepsilon$ and $\psi$ as defined above.
So, there is $u\in\cball[\varepsilon]{x}\subseteq G_n$ satisfying
\begin{equation}\label{rarirurero}
    \abs{\psi(u)-\psi(x)}\geq\frac{\lambda\varepsilon}{8}=\frac{\lambda \varepsilon_n}{8}.
\end{equation}
We want to show, that $\psi(v)=\varphi_n(v)$ for any $v\in\cball[\varepsilon]{x}$.
If $t\in\complement{G_n}$, 
then $\mtr{t}{x}\geq\frac{c}{a^n}$. But $\mtr{x}{v}\le\varepsilon=\frac{\alpha}{a^n}$ and $c-\alpha\ge 1$. Hence
\[
    \mtr{t}{v} \geq \mtr{t}{x}-\mtr{x}{v}\geq\frac{c}{a^n}-\frac{\alpha}{a^n}=\frac{c-\alpha}{a^n}\geq\frac{1}{a^n}.
\]
On the other hand, $\mtr{v}{s}\leq\frac{1}{a^n}$ for $s\in S_n$, 
since $v\in G_n$ and $S_n$ is a $\frac{1}{a^n}$-dense set in $G_n$.
Taking the above estimations into account, we get
\begin{align*}
    \psi(v)=d(v,S_n) &= \inf_{s\in S_n}\mtr{s}{v} \\
    & = \min\left\{\inf_{s\in S_n}\mtr{v}{s},\inf_{t\in\complement{G_n}}\mtr{v}{t}\right\} \\
    & = \inf_{t\in T_n}\mtr{v}{t} = d(v,T_n)=\varphi_n(v).
\end{align*}
In particular, for $v=x$ we have $\psi(x)=\varphi_n(x)$
and for $v=u$, $\psi(u)=\varphi_n(u)$.
Now \eqref{rarirurero} implies, that
\begin{equation}\label{sasisuseso}
    \abs{\varphi_n(u)-\varphi_n(x)}\geq\frac{\lambda\varepsilon_n}{8}.
\end{equation}
Using Proposition~\ref{prop:TWfunction_is_Lipschitz}$(iii)$ we get the inequality
\begin{equation}\label{eq:s_n_estimate}
    \abs{s_n(x)-s_n(u)}\leq\frac{b^n}{b-1}\mtr{x}{u}\leq\frac{b^n}{b-1}\varepsilon_n=\frac{1}{b-1}b^n\varepsilon_n.
\end{equation}
By Proposition~\ref{prop:TWfunction_is_Lipschitz}$(ii)$ we get
\begin{align}\label{eq:r_n_estimate}
    \abs{r_{n+1}(x)-r_{n+1}(u)}&\leq\abs{r_{n+1}(x)}+\abs{r_{n+1}(u)}\nonumber \\
    & \leq 2\frac{b^{n+1}}{(a-b)a^n}=\frac{2b}{\alpha(a-b)}b^n\varepsilon_n.
\end{align}
Using \eqref{sasisuseso}, \eqref{eq:r_n_estimate} and \eqref{eq:s_n_estimate} we can now estimate
\begin{align*}
    \abs{f(x){-}f(u)} &= \abs{\big(s_n(x)+b^n\varphi_n(x)+r_{n+1}(x)\big)-\big(s_n(u)+b^n\varphi_n(u)+r_{n+1}(u)\big)} \\
    &\geq b^n\abs{\varphi_n(x)-\varphi_n(u)}-\abs{s_n(x)-s_n(u)}-\abs{r_{n+1}(x)-r_{n+1}(u)} \\
    &\geq \frac{\lambda}{8}b^n\varepsilon_n-\frac{1}{b-1}b^n\varepsilon_n-\frac{2b}{\alpha(a-b)}b^n\varepsilon_n \\
    &= \left(\frac{\lambda}{8}-\frac{1}{b-1}-\frac{2b}{\alpha(a-b)}\right)b^n\varepsilon_n \\
    &=\gamma b^n\varepsilon_n = \frac{\gamma}{a}b^n\varepsilon_{n-1}>\frac{\gamma}{a}b^nr,
\end{align*}
where $\gamma=\frac{\lambda}{8}-\frac{1}{b-1}-\frac{2b}{\alpha(a-b)}$.
Hence,
\[
    \Lip^rf(x)\geq\frac{\abs{f(x)-f(u)}}{r}\geq \frac{\gamma}{a}b^n, \text{ for any }r<\delta.
\]
Since $\lim\limits_{r\to 0+}n(r)=\infty$, $b>1$ and $\gamma>0$, we have
\[
    \lip f(x)=\liminf_{r\to 0^{+}}\Lip^rf(x)\geq\lim_{r\to 0^{+}}\frac{\gamma}{a}b^{n(r)}=\infty.
\]
Thus, we have shown that $\lip f(x)=\infty$ for arbitrarily chosen $x\in A$.
This means, that $A\subseteq\ell^{\infty}(f)$. But the same reasoning as in
theorem \ref{thm:first_problem} can be applied here to show that 
$\bL^{\infty}(f)\subseteq A$. 
To summarize, we have proven, that
\[
    A\subseteq\ell^{\infty}(f)\subseteq L^{\infty}(f)\subseteq\bL^{\infty}(f)\subseteq A,
\]
hence $A=\ell^{\infty}(f)=L^{\infty}(f)=\bL^{\infty}(f)$.
\end{proof}

\begin{thm}
    Let $X$ be a hermetic metric space and $A\subseteq X$.
    Then, following conditions are equivalent:
    \begin{enumerate}[label=\normalfont{(\roman*)}]
        \item $A$ is a closed set which does not contain isolated points of $X$,
        \item there exists continuous function $f\colon X\to\bR$
        such that \[\ell^{\infty}(f)=\bL^{\infty}(f)=L^{\infty}(f)=A.\]
    \end{enumerate}
\end{thm}

The question of whether the assumption of hermeticity of $X$ in the above theorem  is essential remains open.


\section{Statements and Declarations}

\textbf{Funding.} 
The research was supported by the University of Silesia Mathematics Department (Iterative Functional Equations and Real Analysis program).

\textbf{Competing Interests.} The authors have no relevant financial or non-financial interests to disclose.

\textbf{Author Contributions.} Whole text of the article was made by the authors. Every cited information was indicated in the references.   

\textbf{Data availability.} All data generated or analysed during this study are included in this published article.

\end{document}